\documentclass{amsart}

\usepackage[latin1]{inputenc}
\usepackage{comment}
\usepackage{qtree,array,youngtab}

\usepackage{color, graphicx}
\usepackage{soul}
\usepackage[T1]{fontenc}

\newcommand{\hcal}[2]{{\mathcal{H}}_{#1}\!\left(#2\right)}
\newcommand{\hcount}[2]{{h}_{#1}\!\left(#2\right)}
\newcommand{\hmul}[2]{{{H}}_{#1}\!\left(#2\right)}

\newcommand{\emptypart}{$\phi$}
\newcommand{\emptypartpart}{{$\substack{()\\\cdots00\DASHsmaller11\cdots\\()}$}}

\newcommand{\DASH}{{\textbf{\Large |}}}
\newcommand{\DASHsmaller}{{\textbf{|}}}

\definecolor{darkgreen}{rgb}{0.,0.5,0.}
\newcommand{\bO}{\textcolor{blue}{0}}
\newcommand{\rO}{\textcolor{red}{0}}
\newcommand{\gO}{\textcolor{darkgreen}{0}}

\newcommand{\bI}{\textcolor{blue}{1}}
\newcommand{\rI}{\textcolor{red}{1}}
\newcommand{\gI}{\textcolor{darkgreen}{1}}

\newcommand{\dixhuit}{18}
\newcommand{\dix}{10}
\newcommand{\quinze}{15}
\newcommand{\onze}{11}
\newcommand{\douze}{12}

\newcommand{\quatorze}{14}
\newcommand{\seize}{16}
\newcommand{\dixsept}{17}
\newcommand{\dixneuf}{19}
\newcommand{\vingt}{20}

\newtheorem{thm}{Theorem}
\newtheorem{lemma}[thm]{Lemma}

\DeclareMathOperator{\lcm}{lcm}
\newcommand{\floor}[1]{{\left\lfloor {#1} \right\rfloor}}

\newcommand{\ud}{\mathrm{d}}
\newcommand{\nn}{\mathbf{n}}

\title{Integrality of hook ratios} 
\date{\today}
\author{Paul-Olivier Dehaye}
\address{D-MATH, ETH Z\"urich, 101 R\"amistrasse, 8092 Z\"urich, Switzerland}

\newtheorem{question}{Question}

\keywords{partitions, hook products, Kneser's theorem, McKay numbers, Beurling-Nyman criterion}

\begin{document}
\begin{abstract}
We study integral ratios of hook products of quotient partitions. This question is motivated by an analogous question in number theory concerning integral factorial ratios. We prove an analogue of a theorem of Landau that already applied in the factorial case. Under the additional condition that the ratio has one more factor on the denominator than the numerator, we provide a complete classification. Ultimately this relies on Kneser's theorem in additive combinatorics.
\end{abstract}

\maketitle
\section{Introduction}
\subsection{Questions}
This paper starts with the following (unsolved) question, a classic in number theory:
\begin{question}
For which vectors of positive integers $(\gamma_k:1 \le k \le K)$ and $(\delta_l:1 \le l \le L)$  is 
\begin{equation}
\label{eqn.alt}
v(n;\vec{\gamma},\vec{\delta}) := \frac{\prod_{k=0}^K \floor{\displaystyle\frac{n}{\gamma_k}}!}{\prod_{l=0}^L\floor{\displaystyle\frac{ n}{\delta_l}}!}
\end{equation}
an integer for all integers $n$?
\label{q.alt}
\end{question}
From now on, we assume $k$ runs between 1 and $K$ and $l$ between 1 and $L$. We call $L-K$ the \emph{height}.
Some easy reductions are immediately in order. We will require $\gamma_k \ne \delta_l$ for all $k$ and $l$. We also assume a \emph{balancing} condition: $\sum_k \frac{1}{\gamma_k} = \sum_l \frac{1}{\delta_l}$. This ensures the growth of $v(n;\vec{\gamma},\vec{\delta})$ is exponential rather than factorial in $n$. Historically, a closely related question is studied more often:
\begin{question}
\label{q.original}
For which vectors of positive integers $(\alpha_k : 1 \le k \le K)$ and $(\beta_l: 1 \le l \le L)$ is 
\begin{equation}
\label{eqn.original}
u(n;\vec{\alpha},\vec{\beta}) := \frac{\prod_{k=1}^K (\alpha_k n)!}{\prod_{l=1}^L(\beta_l n)!}
\end{equation}
an integer for all integers $n$?
\end{question}
We will use footnotes to provide information in the context of Question~\ref{q.original}, but these can be ignored  for understanding our results\footnote{In the setting of Question~\ref{q.original}, the \emph{balancing} condition requires $\sum_k \alpha_k = \sum_l\beta_l$.}: they are only there to provide links with existing literature.

Our main goal is to study a generalization of Question~\ref{q.alt}, replacing $n$ with a partition.

Consider a partition $\lambda$, or rather its Young diagram.  Let $\hcal{}{\lambda}$ be the multiset of values of its hooks, and $\hmul{}{\lambda}$ the product of these values, sometimes called the hook product (see Section~\ref{sec.def} for definitions).

Let 
\begin{equation}
\hcal{r}{\lambda} := \left\{ \frac{h}{r}: h \in \hcal{}{\lambda}| h \!\!\mod r = 0\right\},
\end{equation}
(with repetitions) and similarly (with repetitions) $\hmul{r}{\lambda}$ the product of these values. We also set $\hcount{r}{\lambda}$ to be the cardinality of $\hcal{r}{\lambda}$. We see that $\hcal{}{\lambda} = \hcal{1}{\lambda}$ and $\hmul{}{\lambda} = \hmul{1}{\lambda}$.

We will now use the notation $\nn$ for the partition $(n)$. To generalize Question~\ref{q.alt}, it is crucial to observe that 
$\hcal{r}{\nn} = \left\{1,\cdots,\lfloor n/r\rfloor\right\}$,
so $\hmul{r}{\nn} = \lfloor n/r \rfloor! $. In other words, if the partition is one-rowed, its hook product is a factorial that appears in Question~\ref{q.alt} (and similarly if it has only one column, of course). This suggests the following generalization, which will be primarily considered in this paper.
\begin{question}
\label{q.hook}
For which vectors of positive integers $\vec{\gamma} = (\gamma_k)$ and $\vec{\delta} = (\delta_l)$  is 
\begin{equation}
v(\lambda;\vec{\gamma},\vec{\delta}) := \frac{\prod_k\hmul{\gamma_k}{\lambda}}{\prod_l\hmul{\delta_l}{\lambda}}
\end{equation}
an integer for all partitions $\lambda$?
\end{question}
It is immediately clear that the condition required by Question~\ref{q.hook} is stronger\footnote{Note the abuse of notation, which is inconsequential as $v(\nn;\vec{\gamma},\vec{\delta})  = v(n;\vec{\gamma},\vec{\delta})$. } than the condition of Question~\ref{q.alt}, since the latter corresponds to $\lambda = \mathbf{n}$ (\textit{i.e.}~one-rowed).

This accounts for some of the motivation of Question~\ref{q.hook}: Question~\ref{q.alt}  admits only a partial answer at the moment, and one could thus hope that additional conditions the pairs $(\vec{\gamma},\vec{\delta})$ need to satisfy will slim the set to classify to a more manageable structure.

\subsection{Motivation}
There is additional motivation for this new question, coming from both number theory and the representation theory of symmetric groups.
\subsubsection{Motivation in number theory}
The problem posed originally by Question~\ref{q.alt}  is connected to very deep questions in number theory, via the following functions:
\begin{equation}
f(x) = f(x;\vec{\gamma},\vec{\delta})  := \sum_{k=1}^K\floor{\frac{x}{\gamma_k}} -\sum_{l=1}^L\floor{\frac{x}{\delta_l}}.
\label{eqn.function}
\end{equation}

The bridge between $f(x;\vec{\gamma},\vec{\delta}) $ and $v(n;\vec{\gamma},\vec{\delta})$  was uncovered by Landau. 
\begin{thm}[\cite{landau}]
The ratio \eqref{eqn.alt} is integral for all $n$ if and only if the function \eqref{eqn.function} is nonnegative for all real positive $x$.
\end{thm}
Under the balancing condition, the function $f$ is easily seen to be periodic, of period dividing $M(\vec{\gamma},\vec{\delta}) := \lcm(\gamma_1,\cdots,\gamma_K,\delta_1,\cdots,\delta_L )$. We will present additional properties\footnote{Yet more properties result from the fact that this condition applies for all \emph{real} $x$: we can perform a linear scaling of $x$ to help on Question~\ref{q.original} instead of Question~\ref{q.alt}. In fact, these rescalings show that the sets of parameters $(\vec{\gamma},\vec{\delta})$ providing a positive answer to Question~\ref{q.alt} and the $(\vec{\alpha},\vec{\delta})$ for Question~\ref{q.original} are in bijection. Define the bijection $\Phi: (\vec{\mu},\vec{\nu}) \rightarrow ((M/{\mu_k}),(M/{\nu_l}))$ with $M =  M(\vec{\mu},\vec{\nu})$. If $(\vec{\alpha}, \vec{\beta})$ satisfies $\gcd(\alpha_1,\cdots,\alpha_K,\beta_!,\cdots,\beta_L)=1$ and is a solution to Question~\ref{q.original}, then $\Phi(\vec{\mu},\vec{\nu})$ is a solution to  Question~\ref{q.alt}. The converse is also true, and $\Phi$ is involutive if the $\gcd$ condition applies.\label{footnote.phi}} of $f$ in the course of the proof of Theorem~\ref{thm.height1hook}, in  Section~\ref{sec.height1hook}.

The importance of functions of type \eqref{eqn.function} comes from the Beurling-Nyman criterion for the Riemann Hypothesis. This asserts that the Riemann zeta function $\zeta(s)$ has  no zeros in the half-plane $\sigma > \frac{p-1}{p}$ if and only if for any $\epsilon >0$ there exists a function $\tilde{f}(x) = \sum_{n=1}^N c_n \floor{\alpha_n x} 
$
(with $\sum c_n \alpha_n =0$ and $0 \le \alpha_n \le 1$ for all $n$) such that 
$$
\left(\int_{1}^\infty  \left|\frac{1-\tilde{f}(x)}{x}\right|^p \ud x\right)^{1/p} < \epsilon.
$$
In other words, the Riemann Hypothesis is true if and only if linear combinations of similar to  \eqref{eqn.function} can successfully approximate the constant function 1 in the $L^2$-norm  ($p=2$ gives $\frac{p-1}{p} = \frac12$).

An example of $\vec{\delta}$ and $\vec{\gamma}$ would be $(30,1)$ and $(2,3,5)$, which lead to the always-integral factorial ratio
\begin{eqnarray}
v(n;(30,1),(2,3,5)) = \frac{\floor{\frac{n}{30}}!\,\,\,n!}{\floor{\frac{n}{2}}!\floor{\frac{n}{3}}!\floor{\frac{n}{5}}!},
\label{eqn.Chratio}
\end{eqnarray}
thereby providing a positive example for  Question~\ref{q.alt}\footnote{Under the bijection $\Phi$ defined in Footnote~\ref{footnote.phi}, this corresponds to the always-integral ratio
$
u(n;(30,1),(15,10,6)) = ((30n)!n!) /((15n)!(10n)!(6n)!),
$
which is actually closer to the ratio used by Chebyshev.}.
Chebyshev used this ratio to prove the bound
$$
0.92 \frac{x}{\log x} \le \pi(x) \le 1.11 \frac{x}{\log x} 
$$
on the prime counting function, after careful analysis of the valuation of \eqref{eqn.Chratio} at all primes, which he knew was positive. Later, Diamond and Erd\"os  \cite{DE} showed that one can use similar techniques to  improve the constant in both of those bounds. It turns out that these improvements lead all the way to values of 1 for both constants, but proving this requires the Prime Number Theorem (which is of course equivalent to $\pi(x) \sim \frac{x}{\log x}$).

Historically, integrality information for ratios such as \eqref{eqn.Chratio} has been used to deliver elementary proofs of several  statements in number theory. It is our hope that generalizations such as Question~\ref{q.hook} will eventually lead to more powerful techniques, exploiting the fact that partitions extend in 2 dimensions (this would give ``more room'' to construct interesting ratios).

The paper \cite{MR1784410} was partly motivated by similar questions: random matrix theory/number theory conjectures (the so-called Keating-Snaith conjectures) had previously led to some unexpectedly-integral ``bi-factorial ratios'', \textit{i.e.}~ratios of products of factorials at consecutive integers $\prod_{i=a}^b i$!.  The authors studied these ratios in details in that paper, and their valuations at primes displayed intriguing fractal-like patterns. In the current paper, we are still interested in integer ratios, but considering a different  construction.

In the context of using these integral hook ratios for number theory proofs, it is worth noting that an important component in the elementary proofs is the analytic continuation for $n!$ afforded by the $\Gamma$ function and the asymptotics given  by the Stirling formula. Similarly, there exists an extensive theory of hook products, substituting for a discrete product over boxes of the diagram an integral over the shape of the diagram. This can also be replaced by an integral against the outline of the diagram (its graph after rotation of the diagram into the Russian  orientation, given by a piecewise linear function of slope either equal to $\pm 1$ or not defined) of the Barnes $G$-function. For partitions of large size, this outline can be approximated by a Lipshitz 1 function (see \cite{KerovBook} for explanations along those lines). This can then be used to derive asymptotics as well as various expansions for the $G$-function are known.

\subsubsection{Motivation in representation theory}
In addition to potential applications in number theory, it is remarkable that the quantities investigated here appear in group theory as well.
Given an $n$, recall that the theory of representations of $\mathcal{S}_n$ in characteristic 0 can be fully understood via the action of the group its Specht modules, which form a complete set of irreducible representations, are indexed by partitions of $n$, and actually realized over $\mathbb{Z}$. 

One can also define them over a field $k$ of  characteristic $p>0$, but things get more complicated as the Specht modules will no longer be simple. Instead, $k\mathcal{S}_n$ will decompose into a sum $\oplus_{i} B_i$ of two-sided ideals, called \emph{$p$-blocks}. One can then decompose $1 = e_1+\cdots + e_r$ in this direct sum (each $e_i$ is central and idempotent).  Given a Specht module $V^\lambda$ associated to the partition $\lambda$, we see that only one of the $V^\lambda \cdot e_i$ can be nontrivial. We say that $V^\lambda$ (and by extension the partition $\lambda$) \emph{belongs} to the corresponding block. 

Nakayama's conjecture (now a theorem of Brauer and Robinson) says that two partitions belong to the same $p$-block iff they have the same $p$-cores, which are given by a combinatorial process described in Section~\ref{sec.def}.

In addition, the $p$-valuation of hook products is related to conjectures on Sylow subgroups of symmetric groups. Assume $G$ is a finite group and let $M_i(p,G)$ denote the set of ordinary irreducible characters of $G$ of degree divisible by $p^i$ but not $p^{i+1}$,  for $i$ a nonnegative integer. The integers $m_i(p,G) := |M_i(p,G)|$ are called the \emph{McKay numbers of $G$ with respect to $p$}.  McKay has conjectured for instance that $m_0(p,G) =  m_0(p,N_G(H))$ if $H$ is a $p$-Sylow of $G$. This has been proved for symmetric groups by Olsson \cite{OlssonMcKay}. 

We remind the reader of the classical hook formula \cite{Sagan}. If $\chi^\lambda$ is the ordinary character of $\mathcal{S}_n$ (we have that $\lambda$ partitions $n$), its degree, will satisfy
$
\deg \chi^\lambda = {n!}/{H(\lambda)}.
$
This quantity is also sometimes called the dimension of $\lambda$, and counts the number of standard Young tableaux of shape $\lambda$. 

We thus see that $m_i(p,\mathcal{S}_n)$ is concerned with divisibility by powers of $p$ of the hook product of $\lambda$ (this is perhaps most clearly expressed in the literature in \cite{Kane}). Question~\ref{q.hook} is similarly concerned with divisibility by $p$ (for all $p$) of ratios of hook products. 

\subsection{Results}
We will actually proceed as for the classification of integral factorial ratios. This requires proving an analogue of   Landau's theorem, with adaptations to our setting.
\begin{thm}
\label{thm.landaugeneral}
Let $(\gamma_k)$ and $(\delta_l)$ be two vectors of integral parameters. Then, the following two statements are equivalent:
\begin{equation} \frac{\prod_k \hmul{\gamma_k}{\lambda}}{\prod_l \hmul{\delta_l}{\lambda}} \text{ is integral for any partition $\lambda$;}\label{eqn.ratios}
\end{equation}
and
\begin{equation}
\sum_k \hcount{\gamma_k}{\mu} - \sum_l \hcount{\delta_l}{\mu} \ge 0 \text{ for any partition $\mu$.}
\label{eqn.counts}
\end{equation}
\end{thm}
If we require both $\lambda $ and $\mu$ to be one-rowed partitions, these two weaker statements are still equivalent and that result is precisely Landau's original theorem.
This theorem  will be used to give easy proofs that would be very intricate otherwise, such as for the following theorem.
\begin{thm}
\label{thm.multinomial}
Let $s$ and $t$ be positive integers. Then, $$
\frac{\hmul{s}{\lambda}}{\left( \hmul{st}{\lambda}\right)^t}$$
is an integer for all partitions $\lambda$.
\end{thm}
This theorem, when restricted to one-rowed partitions, of course reduces to the integrality of balanced multinomial coefficients.

Despite this good start, we are unable to give a complete classification of integral hook ratios. It is worth observing that the classification of integral factorial ratios is also incomplete. 

For factorial ratios, in the special case of height 1, additional information is available: Bober developed a complete classification in his thesis \cite{bober1}. We will present part of his work in Section~\ref{sec.height1}. The final result includes a relatively complex list of parameters, involving 3 infinite families and 52 sporadic cases (see Section~\ref{sec.factorialheight1}). In contrast, we present here the counterpart of this classification, but this time for hook ratios. The resulting list will reduce to essentially one case. 

\begin{thm}
\label{thm.height1hook}
Assume $\vec{\gamma} $ and $\vec{\delta}$ are such that $v(\lambda;\vec{\gamma},\vec{\delta})$ is integral for all $\lambda$, under the balancing condition that $\sum_k \frac{1}{\gamma_k} = \sum_l \frac{1}{\delta_l}$. Assume further that this hook ratio is of height 1, \emph{i.e.}~$L-K  = 1$. 

Then,  we actually have $K=1$ and thus $L=2$, and $2\gamma_1 = \delta_1 = \delta_2$.
\end{thm}

\section{Definitions}
\label{sec.def}
A \emph{partition} $\lambda=(\lambda_1,\cdots,\lambda_d)$ is a weakly decreasing, finite sequence of positive integers, called its \emph{parts}.  The \emph{size} of the partition $\lambda$, denoted $|\lambda|$, is the sum of its parts. We actually prefer to think of a partition in terms of its \emph{Young diagram}. This allows to consider  its  \emph{boxes}, and to define a \emph{hook} based at the box $\square$ as the set of boxes exactly to the right, or exactly below $\square$, or $\square$ itself. This associates to each box an integer, the cardinality of the hook of that box, called the \emph{ hook length}. We give some illustrations in Figure~\ref{fig.hooks}.

\begin{figure}[htbp]
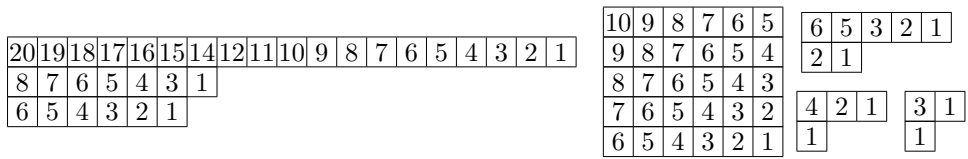

\begin{center}
\Yvcentermath1
\begin{tabular}{@{}cc@{}c@{}}
$\young(\vingt\dixneuf\dixhuit\dixsept\seize\quinze\quatorze\douze\onze\dix987654321,8765431,654321)$
&
$\young(\dix98765,987654,876543,765432,654321)$
&
\begin{tabular}{c}
$\young(65321,21)$
\vspace{.1in}
\\
$\young(421,1)$
\hspace{.00in}
$\young(31,1)$
\end{tabular}
\end{tabular}
\end{center}
\caption{The Young diagrams of the partitions $(18,7,6)$, $(6,6,6,6,6)$, $(5,2)$, $(3,1)$ and $(2,1)$, together with the hook lengths of the corresponding boxes.\label{fig.hooks}}
\end{figure}

Given a diagram $\lambda$, we can follow its \emph{01-sequence} as the (bidirectional) sequence of steps taken when following the outline of the partition, starting at the bottom and ending to the right (see Figure~\ref{fig.01} for an example).  A vertical step (up) is denoted by a 0, a horizontal step (right) by a 1. The 01-sequence $(x_i)$ of a partition, which is indexed by integers, is thus eventually 0 for negative $i$ and eventually 1 for positive $i$. We have only defined the sequence up to a shift so far, so we follow the additional  convention that $|\{x_i =0: i \ge 0\}|= |\{x_i =1: i<  0\} |$. There is then a unique 01-sequence associated to each partition, and conversely this sequence determines the partition. We write $x(\lambda)$ for the 01-sequence of $\lambda$. Figure~\ref{fig.sequence} gives a particular example.

\begin{figure}[htbp]
  \begin{center}
    \includegraphics[width=0.78\textwidth]{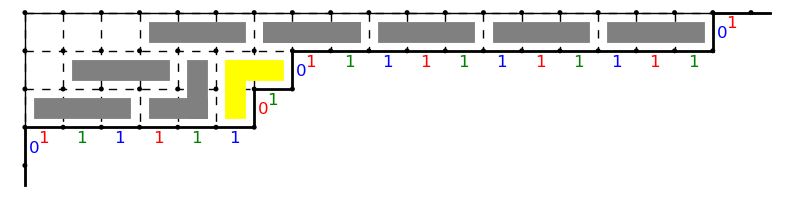}
        \caption{The Young diagram of the partition $\lambda = (18,7,6)$ together with its 01-sequence $\cdots\bO                 \rI\gI\bI \DASH \rI\gI\bI \rO\gI\bO  \rI\gI\bI \rI\gI\bI \rI\gI\bI \rI\gI\bO \rI\cdots$ (the $\DASH$ indicates the unique location where the number of 1s to the left equals the number of 0s to the right, which is between index -1 and 0 in the 01-sequence; visually this corresponds to the main diagonal in the Young diagram). Notice the mnemonic ``RGB'' for the colourings, starting after the $\DASH$ mark.
        The grey shapes indicate a sequence of 3-hooks that can be removed from $\lambda$, to finally obtain the 3-core $(3,1)$ of $\lambda$. One could start with the yellow shape for instance, and then make various choices for the order of the removals. The process always ends with the same $3$-core, $(3,1)$ (see Figure~\ref{fig.hooks}).
        Quotients will be $\cdots\rO\rI\DASH\rI\rO\rI\rI\rI\rI\rI\cdots$ (partition $(2)$), $\cdots\gO\DASH\gI\gI\gI\gI\gI\gI\gI\cdots$ (empty partition) and $\cdots\bO\bI\bI\DASH\bO\bI\bI\bI\bO\bI\cdots$ (partition $(5,2)$). Remark that the yellow shape connects two labels in the 01-sequence that are blue. It thereby corresponds to a specific box of the quotient partition $(5,2)$ (see Figure~\ref{fig.hooks} and compare with Equation~\eqref{eqn.essential}).\label{fig.sequence}}
    \label{fig.01}
  \end{center}
\end{figure}


Let $p$ be a positive integer, not necessarily prime. We say that a partition $\lambda$ is a \emph{$p$-core} if it has no hook of length divisible by $p$ (it is actually equivalent to ask for the partition not to have a hook of length exactly $p$). For instance, $(3,1)$ is a 3-core as can be seen from Figure~\ref{fig.hooks}.

In terms of the 01-sequence $(x_i)$, a hook of length $p$ in $\lambda$ corresponds to a pair $x_{i_0} = 1$, $x_{i_0+p} = 0$. Swapping these two entries in the sequence $(x_i)$ gives the 01-sequence of a partition of size decreased by $p$, due to the removal of a hook of length $p$. Starting with one partition $\lambda$, we can swap pairs at distance $p$ and iterate this procedure until no more $p$-hook can be found. The partition associated to this final 01-sequence is actually independent from the ordering of the hook removals (or the ordering of swaps of a $\cdots1\cdots 0\cdots$ into a $\cdots0\cdots 1\cdots$), and called the \emph{$p$-core of $\lambda$}, which we denote $\lambda^{[]}$. We see in Figure~\ref{fig.01} that the 3-core of the partition $(18,7,6)$ is $(3,1)$. 

We now define the $p$-quotients of the partition $\lambda$, for $p$ still not a prime. We start with $ (x_i) := x(\lambda) $, and consider the $p$-tuple of subsequences given by $((x_{p i +j} ): 0\le j \le p-1)$. For each $0\le j\le p-1$, this defines a 01-sequence and thus an associated partition, which we call the $j^{\text{th}}$ \emph{$p$-quotient} of $\lambda$ and denote $\lambda^{(j)}$. We  obtain a $p$-tuple of $p$-quotients for $\lambda$.  Figure~\ref{fig.01} presents a computation of quotients.

These quotient partitions have the  property, essential for us, that
\begin{equation}
\label{eqn.essential}
\hcal{p}{\lambda} = \underset{0 \le j \le p-1}{\cup}  \hcal{}{\lambda^{(j)}},
\end{equation}
with the union taken with multiplicity (Figure~\ref{fig.hooks} has the hook lengths for the partition that appears in Figure~\ref{fig.01} and its quotients). In other words, the union of the hook lengths of the $p$-quotients of $\lambda$ give the hook lengths of $\lambda$ that are divisible by $p$, divided by $p$. In particular, this implies 
$\hcount{p}{\lambda} = \sum_{0 \le j \le p-1} \hcount{}{\lambda^{(j)}}$ or even $\hcount{pk}{\lambda} = \sum_{0 \le j \le p-1} \hcount{k}{\lambda^{(j)}}$.

For each positive integer $p$, we have so far defined a map, 
$
\{\text{partitions} \} \rightarrow \{ p-\text{cores} \} \times \{ \text{partitions} \}^p,
$
called the \emph{Littlewood decomposition at $p$}, and the construction actually shows that this map is a bijection. This can be thought of as a generalization of Euclidian division for integers. We refer the reader to \cite{Macdonald} for more information.
The Littlewood decomposition also satisfies
\begin{equation}
\label{eqn.Littlewood_counts}
|\lambda| = |p-\text{core}(\lambda)| + p \sum_{j=0}^{p-1} |\lambda^{(j)}|,
\end{equation}
which will be useful for us (we repeat that there is no requirement for $p$ to be prime).

Integers admit representations in any base  $p$ (assuming $p>1$). Similarly, given a partition $\lambda$, we will iterate the Littlewood decomposition at $p$. We define recursively $\lambda^{(i_1,\cdots,i_d)} = \lambda^{(i_1,\cdots,i_{d-1})(i_{d})}$. For each tuple in $\{0,1,\cdots,p-1\}^d$, this defines a partition. Note that by convention $\lambda^{()} = \lambda$. We also define $\lambda^{[i_1,\cdots,i_d]} = \lambda^{(i_1,\cdots,i_d)[]}$ (\textit{i.e.}~the core of that particular quotient).

We see these two definitions as two mappings, from $S= \{\{0,\cdots,p-1\}^d:d \in \mathbb{N}\}$ to either partitions or $p$-cores. We see elements of $S$ as indices for the vertices of the $p$-ary infinite rooted tree\footnote{A \emph{$p$-ary infinite rooted tree} is an infinite connected graph with no cycle such that each vertex is of degree $p+1$, with the exception of a distinguished vertex, called its \emph{root}, which is of degree $p$. }: recursively, the index $(i_1,\cdots,i_d)$ corresponds to the $i_d^\text{th}$ child of the vertex indexed by $(i_1,\cdots,i_{d-1})$, with the index $()$ mapped to the root. This means $\lambda^{(i_1,\cdots,i_d)}$ (resp.~$\lambda^{[i_1,\cdots,i_d]}$) provides a labelling of the vertices of the $p$-ary infinite rooted tree by partitions (resp.~$p$-cores), called the \emph{$p$-quotient tower of $\lambda$} (resp.~\emph{$p$-core tower of $\lambda$}). We can deduce from  \eqref{eqn.Littlewood_counts} that all but finitely many labels of the $p$-core tower or the $p$-quotient tower will be trivial partitions.  An example is presented in Figure~\ref{fig.tree}.

\newcolumntype{C}{ >{\centering\arraybackslash} m{5cm} }
\begin{figure}[hbp]
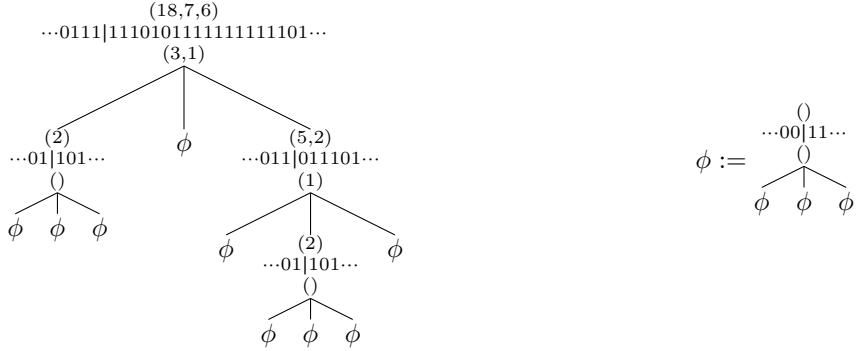

\begin{center}
\begin{tabular}{Cp{2cm}C}
    \Tree [.{$\substack{(18,7,6) \\ \cdots0111\DASHsmaller1110101111111111101\cdots \\ (3,1)}$} 
    [.{$\substack{(2)\\\cdots01\DASHsmaller101\cdots\\()}$} {\emptypart} {\emptypart} {\emptypart} ] 
    [.{\emptypart} ] 
    [.{$\substack{(5,2)\\\cdots011\DASHsmaller011101\cdots\\(1)}$} {\emptypart} [.{$\substack{(2)\\\cdots01\DASHsmaller101\cdots\\()}$} {\emptypart} {\emptypart} {\emptypart} ]  {\emptypart} ] 
    ] &&
{$\phi:=\substack{\Tree [.{\emptypartpart} {\emptypart} {\emptypart} {\emptypart} ]}$}
    \end{tabular}
    \end{center}
    \caption{\label{fig.tree}
    We give the 3-ary infinite rooted tree and three different labellings of the vertices of that tree associated to the partition $\lambda = (18,7,6)$. Selecting the first line at each vertex gives the $3$-quotient tower of the $\lambda$. Selecting the third line at each vertex gives the $3$-core tower of $\lambda$. The second line gives the 01-sequence of the label in the $3$-quotient tower of $\lambda$. Of course, this tree is infinite: to write the trivial entries more compact, we rely on the recursive definition on the right, which also happens to be these same labellings for the partition $\lambda = ()$. 
    The children of a node can be computed via the 01-sequence, by splitting the 0s and 1s according to the colouring conventions explained in Figure~\ref{fig.01}.}
\end{figure}

We can now state the following lemma, where $v_p(n)$ for the largest exponent $\alpha$ such that $p^\alpha$ divides $n$.
\begin{lemma} 
\label{lemma.tower}
Let $p$ be a prime. Then, 
\begin{equation}
v_p(\hmul{}{\lambda}) = \sum_{d=1} d \sum_{(q_i) \in \{0,\cdots,p-1\}^d }  \left|\lambda^{[q_1,\cdots,d_d]}\right|.
\label{eqn.valuation}
\end{equation}
\end{lemma}
\begin{proof}
The proof of this lemma is an easy consequence of the stronger statement that for $d$ a positive integer 
\begin{equation}
\left|\left\{\square \in \lambda: v_p(h(\square))=d\right\}\right| =  \sum_{(q_i) \in \{0,\cdots,p-1\}^d }  \left|\lambda^{[q_1,\cdots,d_d]}\right|,
\label{eqn.stronger}
\end{equation}
which itself results from the recursive definition of the $p$-core tower and Equation~\eqref{eqn.essential} for the base case.
\end{proof}

\textbf{Remark. }If $p$ is not prime, \eqref{eqn.stronger} remains true, but fails to imply \eqref{eqn.valuation}. Consider the example of $p=6$ and $\lambda = (6)$. In that case, the $6$-core tower of $(6)$ consists of a single nontrivial label, by the partition of size 1, at a a node adjacent to the root (so the RHS evaluates to 1). Meanwhile, we have $\hmul{}{\lambda}=720$, so the LHS of \eqref{eqn.valuation} evaluates to 2. This extra factor of 6 actually comes from the presence of hooks of length 2 and 3, with no direct consequence on the $6$-core tower.

\section{Proofs of Theorem~\ref{thm.landaugeneral} and \ref{thm.multinomial}}
We immediately prove Theorem~\ref{thm.landaugeneral}, the generalization of Landau's theorem.

\begin{proof}[of Theorem~\ref{thm.landaugeneral}]

\textbf{\eqref{eqn.ratios} $\Rightarrow$ \eqref{eqn.counts}:} We prove the statement by contraposition. Assume that for some $\mu$ the sum in \eqref{eqn.counts} is negative, \emph{i.e.~}$\sum_i \hcount{r_i}{\mu} - \sum_j \hcount{s_j}{\mu} < 0.$ We will construct a partition $\lambda$ that fails to satisfy Equation~\eqref{eqn.ratios}. 

Take a prime $p$, larger than any hook length in $\mu$. Consider now  the partition $\lambda$ whose core is the empty partition and whose quotient consists of $p$ copies of $\mu$ (many other choices are possible). Then the valuation of the LHS of \eqref{eqn.ratios} at $p$ is negative (and at least -$p$), our contradiction.

\textbf{\eqref{eqn.counts} $\Rightarrow$ \eqref{eqn.ratios}:} Pick a prime $p$. Our goal is to prove that the valuation of \eqref{eqn.ratios} at $p$ is positive. Because \eqref{eqn.ratios} is a quotient, this valuation is given as a linear combination of the counts of hooks of $\lambda$ divisible by the $\gamma_k$ and $\delta_l$. 
This is thus a linear combination of the valuations at $p$ of the hook product of $\nu$, for $\nu$ running through the various $\gamma_k$ and $\delta_l$ quotients of $\lambda$. Each such valuation can be expressed, thanks to Lemma~\ref{lemma.tower}, as a(n effectively finite) sum over the labels in the $p$-core tower of $\nu$.  We now swap this linear combination over $\nu$ and the infinite sum over label positions in a $p$-core tower. We then get a sum over positions in the $p$-core tower of an expression of the form \eqref{eqn.counts} ($\mu$ runs through the $p$-cores appearing in the $p$-core tower of $\nu$). Since each such expression is positive by assumption, the overall sum is also positive.
\end{proof}

Once Theorem~\ref{thm.landaugeneral} is proved, Theorem~\ref{thm.multinomial} is indeed trivial to prove:

\begin{proof}[of Theorem~\ref{thm.multinomial}]
Thanks to Theorem~\ref{thm.landaugeneral}, we know we only need to show
\begin{equation}
\hcount{s}{\lambda} - t \cdot \hcount{st}{\lambda} \ge 0 
\end{equation}
for all $\lambda$, $s$ and $t$. This follows immediately from applying \eqref{eqn.Littlewood_counts} to each of the $s$-quotients of $\lambda$.
\end{proof}

\textbf{Remark.}  The same proof works for $v(\lambda;(x),\vec{\delta})$, as long as $x$ divides each of the $\delta_l$s. This case corresponds under $\Phi$ to a multinomial coefficient $u(n;(\sum_l \beta_l),\vec{\beta})$, but not all $u$ of that form map back to $v$ that answer Question~\ref{q.hook} positively. Consider $\vec{\beta} = (9,6,3,3,3)$, which then gives $v(\lambda;(3),(8,12,24,24,24))$, which is non-integral as soon as $\lambda$ is a 3-core with a hook of length 8 (by Theorem~\ref{thm.landaugeneral}), for instance $(6,4,2)$.
\subsection{Explicit construction of non-integral ratios}
\label{sec.nonintegral}
The proof of Theorem~\ref{thm.landaugeneral} actually constructs example partitions leading to non-integral ratios: it translates a $\mu$ failing Condition~\eqref{eqn.counts} into a $\lambda$ failing Condition~\eqref{eqn.ratios}.

Consider for instance the parameters $(\gamma_k) = (1,30)$ and $(\delta_l) = (2,3,5)$.  Take the partition $\mu = (6,6,6,6,6)$ and evaluate Condition~\eqref{eqn.counts} (maybe with the help of Figure~\ref{fig.hooks}). We get $30+0-15-10-6 = -1$ (a one-rowed $\mu$ would have given 0, see Theorem~\ref{thm.bober}, page \pageref{thm.bober}). Now that we have found a $\mu$ that fails \eqref{eqn.counts}, we construct a $\lambda$ that fails \eqref{eqn.ratios}. The largest hook length in $\mu$ is 10. We thus pick $p= 11$, and compute that the partition with empty core and quotients given by 11 copies of $\mu$ is $(66^{55})$, 
where \eqref{eqn.ratios} evaluates to the \emph{fraction}
$$
\frac{\substack{2^{60} 3^{53} 5^9 7^{35} 19^{12}  23^23  29^{29}  31^{31}  37^{37}  41^{41}  43^{43}  47^{47} 53^{53}  59^{55}  61^{55}  67^{54}  71^{50}  73^{48}  79^{42}\\  83^{38}  89^{32}  97^{24}  101^{20}  103^{18}  107^{14}  109^{12}  113^8}}{11^{11}}.
$$
%
%

In order to show a set of parameters fails at Question~\ref{q.hook}, it is thus efficient  to exhibit a $\mu$ that fails Equation~\eqref{eqn.counts}.  For the classification at height 1, we will produce $\mu$s more systematically, of a very slim shape, and that are actually hooks themselves (\textit{i.e.} their diagrams will not contain a $2\times 2$ rectangle).
\section{Ratios of height 1}
To better appreciate Theorem~\ref{thm.height1hook}, we first present the situations for integral factorial ratios.
\label{sec.height1}
\subsection{Classification of integral factorial ratios}
\label{sec.factorialheight1}
Based on an interpretation due to Rodriguez-Villegas of Question~\ref{q.original} in terms of hypergeometric functions (Question~\ref{q.original} succeeds if and only if the hypergeometric function $\sum_n u(n;\vec{\alpha},\vec{\beta}) z^n$ is algebraic), Bober \cite{bober1} was able to fully classify parameters providing a positive answer to Question~\ref{q.original}, if the height $L-K$ is assumed to be 1. For completeness, we now present this classification.

\begin{thm}[\cite{bober1}]
\label{thm.bober}
Let $L=K+1$. Assume $\alpha_k \ne \beta_l$ for all $k,l$, that $\sum_k \alpha_k = \sum_l \beta_l$ and that 
$\gcd(\alpha_1,\cdots,\alpha_k,\beta_1,\cdots,\beta_l)=1.$
Then, $u(n;\vec{\alpha},\vec{\beta})$ is an integer for all $n$ if and only if 
\begin{enumerate}
\item $(\vec{\alpha},\vec{\beta})$ takes one of the three forms: $((x+y),(x,y))$ for $gcd(x,y)=1$, $((2x,y),(x,2y,x-y))$ for $gcd(x,y)=1$ and $x>y$ or $((2x,2y),(x,y,x+y))$ for $gcd(x,y)=1$.
\item $(\vec{\alpha},\vec{\beta})$ is one of 52 sporadic parameter sets, for instance $((30,1),(15,10,6))$, which  corresponds  under $\Phi$ to $(\vec{\gamma},\vec{\delta}) = ((30,1),(2,3,5))$. 
\end{enumerate}
\end{thm}

\subsection{Integral hook ratios}
\label{sec.height1hook}
Since the conditions presented in Question~\ref{q.hook} are stronger than in Question~\ref{q.alt}, we expect a subset of the (image under $\Phi$ of the) parameter list presented in Theorem~\ref{thm.bober} to provide a positive answer to Question~\ref{q.hook}. In fact, we can  show that this stronger condition slims down entirely  the infinite families and the sporadic cases, with the exception of one lone set of parameters. This is Theorem~\ref{thm.height1hook}, which we now prove.

\begin{proof}[of Theorem~\ref{thm.height1hook}]
We start by recalling parts of the theory for integral factorial ratios, based on the functions $f(x;\vec{\lambda},\vec{\gamma})$.
In the balanced case, for integral factorial ratios of height $L-K$, Landau's theorem and Bober's theory lead  to a $f(x;\vec{\gamma},\vec{\delta})$ with values in $\{0,1,\cdots,L-K\}$ (this is property \textbf{(P0)}). In addition, some extra properties are satisfied (these are easy to prove from the definition in \eqref{eqn.function}):  
\begin{description}
\item[(P1)] $f$ is periodic, of period $P$ dividing $M=M(\vec{\gamma},\vec{\delta}) = \lcm(\gamma_1,\cdots,$ $\gamma_K,$ $\delta_1,\cdots,$ $\delta_L) $;
\item[(P2)] for small enough $\epsilon>0$, we have $f(x)  + f(\epsilon+M-1-x) = L-K$ for all real $x$;
\item[(P3)] $f(M-\epsilon) = f(P-\epsilon) = L-K$, for $\epsilon>0$ small enough.
\end{description}

Property \textbf{(P3)} is  a direct consequence of \textbf{(P2)}, and \textbf{(P0)} is derived slightly more streinuously from \textbf{(P2)}.

By the discussion in Section~\ref{sec.nonintegral}, for a given pair $(\vec{\gamma},\vec{\delta})$ to fail Question~\ref{q.alt}, we need  a $\mu$ such that 
\begin{equation}
\sum_k \hcount{\gamma_k}{\mu} - \sum_l \hcount{\delta_l}{\mu} < 0.\label{eqn.hookcount}
\end{equation}
From this $\mu$, one can then construct (much bigger) examples of partitions  $\lambda$ failing Question~\ref{q.hook} outright. 

Since $h_i(\mu)$ counts hooks of $\mu$ divisible by $i$, an equivalent condition to \eqref{eqn.hookcount} is 
\begin{equation}
\sum_{\square \in \mu} g(h(\square))<0,
\label{eqn.equivalent}
\end{equation}
if $g(x) = g(x;\vec{\gamma},\vec{\delta}):= \sum_k \mathbf{1}{(\gamma_k \text{ divides }x)} - \sum_l \mathbf{1}{(\delta_l \text{ divides }x)}$ (and $\mathbf{1}(\text{condition})$ is 1 or 0 depending on whether the condition is true or false). Crucially, we have (for nonnegative integral $x$)
\begin{equation}
\sum_{y=1}^x g(y;\vec{\gamma},\vec{\delta})= f(x;\vec{\gamma},\vec{\delta}).\label{eqn.integration}
\end{equation}
We are now only thinking of ratios of height 1. Therefore, by \textbf{(P0)}, the values of $f$ are either 0 or 1. Hence, the values of $g$ are either 0,1 or -1. It is clear that $g$ is periodic as well, of period $P$. 

Because Question~\ref{q.hook} is stronger than Question~\ref{q.alt}, we can assume all the properties \textbf{(P0)}-\textbf{(P3)} for the functions $f$ and $g$. These requirements come from considering partitions with just one row. 

We now consider partitions of hook shape only. Take $\mu = (1+a,1,\cdots,1)$, with $l$ repetitions of 1 at the end ($a$ stands for arm, $l$ for leg). This partition has hooks of length $1,2,\cdots,l$, $1,2,\cdots,a$, and $a+l+1$.
Therefore, Condition~\eqref{eqn.equivalent} for finding a $\mu$ reduces to
\begin{equation}
\sum_{i=1}^a g(i) +\sum_{i=1}^l g(i) +g(a+l+1) = f(a)+f(l)+f(a+l+1)-f(a+l) < 0.
\label{eqn.conditionf}
\end{equation}
Given the possible values of $f$, we are left with no choice. We must have 
\begin{equation}
\label{eqn.setup}
f(a)=f(l)=f(a+l+1)=0 \text{ and }f(a+l)=1.
\end{equation}

We will soon need Kneser's theorem, a generalization of the Cauchy-Davenport theorem. Let $G$ be an abelian group. If $A \subseteq G$, define the \emph{stabilizer} of  $A$ to be $\mathcal{S}(A) = \{ g \in G | A+g = A\}.$ Clearly, $0 \in \mathcal{S}(A) \le G$.  Let $A,B \subseteq G$, finite and nonempty sets, with $S= \mathcal{S}(A+B)$. Then, Kneser's theorem affirms that
$|A+B| \ge |A+S| + |B+S| -|S|$.

Let $P$ be the period of $f$, which we know divides $M$. Define $A_i = \{0,\cdots,P-1\} \cap \{x  | f(x) = i\}$. We have thanks to \textbf{(P2)} that $|A_0| = |A_1| = P/2$. 
We will now take $G = \mathbb{Z}/P\mathbb{Z}$, with $A=B=A_0$. By the definition of the period, we know that $|S|=1$. We then have that $|A_0+A_0| \ge P-1$. 

Set $Y = \{ y : y \in \{0,\cdots,P-1\} | g(y) = 1 \text{ and } g(y+1) = 0\}$. By \textbf{(P3)}, we know that $P-1 \ \in Y$. To solve \eqref{eqn.setup}, we thus need to find $a$ and $l$ such that $a+l \in Y \cap (A_0+A_0)$ (this equation is in $\mathbb{Z}/P\mathbb{Z}$, not merely $\mathbb{Z}$). Because of the cardinality bounds $|Y|\ge 1$ and $|A+A|\ge P-1$, this argument can only fail to provide a $\mu$ if both bounds are actually equalities. In that case we would thus have $A_0+A_0 = \{0,\cdots,P-2\}$, with a unique drop for $f$ along one period (this is the definition of $Y$) occurring  at $P-1 ,P$ (by \textbf{(P3)}) and of height 1 (by \textbf{(P0)}). Since $f$ has only one drop, it can also have only one increase (since its values can only be 0 or 1). The location of that unique increase can then be determined to be at $P/2-1,P/2$ (due to \textbf{(P0)}), and this gives the whole statement.
\end{proof}

\textbf{Example. } We provide another $\mu$ than $(6,6,6,6,6)$ for $\vec{\gamma} = (30,1), \vec{\delta} = (2,3,5)$ (see Section~\ref{sec.nonintegral}). In that case, the values taken by $f(x;(30,1),(2,3,5))$ are
$$
\begin{array}{c|p{0.0025\textwidth}p{0.0025\textwidth}p{0.0025\textwidth}p{0.0025\textwidth}p{0.0025\textwidth}p{0.0025\textwidth}p{0.0025\textwidth}p{0.0025\textwidth}p{0.0025\textwidth}p{0.0025\textwidth}p{0.0025\textwidth}p{0.0025\textwidth}p{0.0025\textwidth}p{0.0025\textwidth}p{0.0025\textwidth}p{0.0025\textwidth}p{0.0025\textwidth}p{0.0025\textwidth}p{0.0025\textwidth}p{0.0025\textwidth}p{0.0025\textwidth}p{0.0025\textwidth}p{0.0025\textwidth}p{0.0025\textwidth}p{0.0025\textwidth}p{0.0025\textwidth}p{0.0025\textwidth}p{0.0025\textwidth}p{0.0025\textwidth}p{0.0025\textwidth}p{0.0025\textwidth}}
x\!\!\!\!  \mod 30&0&1&2&3&4&5&6&7&8&9&10&11&12&13&14&15&16&17&18&19&20&21&22&23&24&25&26&27&28&29&$\cdots$\\\hline
f&0& 1& 1& 1& 1& 1& 0& 1& 1& 1& 0& 1& 0& 1& 1& 0& 0& 1& 0& 1& 0& 0& 0& 1& 0& 0& 0& 0& 0& 1&$\cdots$
\end{array} 
$$
We see that $P = 30$, $Y= \left\{ 5,9,11,14,17,19,23,29\right\}$ and $A_0 = \{0, 6, 10, 12, 15, 16, 18, 20, 21,$ $22,$ $24,$ $25,$ $26, 27, 28\}$, with $A_1$ the complement to $A_0$. One can check that $A_0 + A_0 = \{0,\cdots,28\}$ (remember this is in $\mathbb{Z}/30\mathbb{Z}$!), with $29$ missing. Because $Y$ is not a singleton, any element of $Y$ different from 29 leads to a $\mu$. Take $9\in Y$. Then,  we solve (non-constructively and $\mod 30$) $9 =  15+24 \in A_0+A_0$. Therefore,  we have found an example partition  $\mu$ that does not satisfy \eqref{eqn.counts}: it has an arm length $a=15$, and leg length $l=24$, \textit{i.e.}~$(16,1,\cdots,1)$, with 24 trailing zeroes. This $\mu$ can be used to construct a $\lambda$ that fails to satisfy \eqref{eqn.ratios}, using the technique of Section~\ref{sec.nonintegral}. 

%

\section{Bounds}
For integral factorial ratios, some bounds are known \cite{bober2} for $K+L$ given the value of $L-K$. These techniques were improved in \cite{Schmerling}. The bounds obtained can be stated in different ways: as an explicit bound on $K+L$ as a function of $L-K$, as an asymptotic bound on the growth of that function, or optimized and explicit for specific values of $L-K$. 

For the small values of $L-K$, these bounds seem to overestimate the actual value of $L+K$: in the case of $L-K=1$ for instance, the exhaustive classification of Bober (Theorem~\ref{thm.bober}) gives a maximal value of 9 for $K+L$ (for one of the 52 sporadic cases). Schmerling's optimized version of the bound when $L-K=1$ gives an upper bound of 43, still far off the truth. For $L-K=2$ or 3, non-exhaustive computer searches point to a similar overshoot of the upper bounds.
 
As Question~\ref{q.hook} is stronger than Question~\ref{q.alt}, these bounds will also apply for integral hook ratios. For completeness we state a general version of the bounds here.

\begin{thm}[\cite{Schmerling}, improving on \cite{bober2}, restated for hook ratios]
Assume $((\gamma_k:1 \le k \le K),(\delta_l:1 \le l \le L))$ are such that $\gamma_k \ne \delta_l$ for all $k,l$, $\sum_k \frac{1}{\gamma_k} = \sum_l \frac{1}{\delta_l}$ and provide a positive answer to Question~\ref{q.hook}.
Then, 
$
(K+L)\le 287 (L-K)^{3.44},
$
and there exists an absolute constant $C>0$ such that
$
(K+L) \le C (L-K)^2 (\log |\log (L-K)|)^2.
$
\end{thm}
\section{Related questions}
Question~\ref{q.alt} is  equivalent to Question~\ref{q.original} for factorial ratios. Question~\ref{q.hook} extends Question~\ref{q.alt} to hook ratios, but Question~\ref{q.original} does not extend in the same way. Indeed, quotients of partitions are given  by the Littlewood correspondence but there is no suitable operation constructing multiples of partitions (at least not with the correct effect on partition sizes). Still, another question  might extend Question~\ref{q.original} rather than Question~\ref{q.alt}.

A second remark concerns the motivation in number theory. In their paper \cite{DE}, Diamond  and Erd\"os need not have an integral factorial ratio $u(n;\vec{\alpha},\vec{\beta})$ for \emph{all} n, but they are satisfied with a sequence $(n_i)$ going to infinity such that  $u({n_i};\vec{\alpha},\vec{\beta})$ is always integral. It would be interesting to consider a similar relaxation to Question~\ref{q.hook} (for instance by restricting the sizes of the partitions considered to be in some subset of the integers).

\section{Acknowledgements}
This paper benefited from conversations with J.~Bober, Prof.~J.~Olsson and indirectly Prof.~Ch.~Bessenrodt.

\bibliographystyle{siam}
\bibliography{references}
\label{sec:biblio}
\end{document}